\documentclass{amsart}
\usepackage{amsmath,amssymb,amsthm,url,comment}
\usepackage{graphicx,color}
\usepackage{pict2e}
\newtheorem{theorem}{Theorem}[section]
\newtheorem{lemma}[theorem]{Lemma} 
\newtheorem{claim}{Claim} 
\newtheorem{proposition}[theorem]{Proposition} 
\newtheorem{conjecture}{Conjecture} 
 
\newtheorem{corollary}[theorem]{Corollary}
\theoremstyle{definition}
\newtheorem{definition}[theorem]{Definition}
\newtheorem{example}[theorem]{Example}
\newtheorem{remark}[theorem]{Remark}
\newtheorem*{remark*}{Remark}

\definecolor{dg}{rgb}{0.1,0.4,0.1}

\numberwithin{equation}{section}

\newcommand{\Z}{\mathbb{Z}}
\newcommand{\R}{\mathbb{R}}

\newcommand{\btw}{\overline{tw}}

\newcommand{\GaussD}{
\raisebox{-4mm}{
\begin{picture}(24,28)
\put(12,14){\circle{24}}
\put(9,0){$\ast$}
\put(4,5){\vector(1,1){17}}
\put(20,5){\vector(-1,1){17}}
\end{picture} 
}
}

\newcommand{\Larrow}{
\raisebox{-4mm}{
\begin{picture}(24,28)
\put(12,14){\circle{24}}
\put(9,0){$\ast$}
\put(24,14){\vector(-1,0){24}}
\end{picture} 
} 
}

\newcommand{\Rarrow}{
\raisebox{-4mm}{
\begin{picture}(24,28)
\put(12,14){\circle{24}}
\put(9,0){$\ast$}
\put(0,14){\vector(1,0){24}}
\end{picture} 
} 
}
\begin{document}

\title[chirally cosmetic surgeries on special alternating knots]{Special alternating knots with sufficiently many twist regions have no chirally cosmetic surgeries}

\author{Tetsuya Ito}
\address{Department of Mathematics, Kyoto University, Kyoto 606-8502, JAPAN}
\email{tetitoh@math.kyoto-u.ac.jp}

\subjclass[2020]{Primary~57K10,57K18,57K31}

\begin{abstract}
We show that a special alternating knot with sufficiently large number (more than $63$) of twist regions has no chirally cosmetic surgeries, a pair of Dehn surgeries producing orientation-reversingly homeomorphic $3$-manifolds. In the course of proof, we provide the optimal upper bounds of the primitive finite type invariants of degree 2 and 3 that solve Willerton's conjecture. 
\end{abstract}

\maketitle
\section{Introduction}

A pair of two Dehn surgeries on the same knot $K$ with inequivalent slopes are called \emph{chirally cosmetic} if they yield \emph{orientation-reversingly} homeomorphic $3$-manifolds. When a knot $K$ is amphicheiral, for any slope $r$, its $r$ surgery and $-r$ surgery are chirally cosmetic. As a more illuminating example, for each integer $m$, $\frac{2p^{2}(2m+1)}{p(2m+1)+1}$ and $\frac{2p^{2}(2m+1)}{p(2m+1)-1}$ surgeries on the $(2,p)$-torus knots are chirally cosmetic (see \cite{ro}, \cite[Appendix]{iis1}).

Currently no other examples of chirally cosmetic surgeries for knots in $S^{3}$ are known. Therefore it is natural to expect that there are no other chirally cosmetic surgeries on knots in $S^{3}$.

\begin{conjecture}
\label{conjecture:main}
A chirally cosmetic surgery on non-trivial knots in $S^{3}$ is either $\pm r$ surgeries of amphicheiral knots, or, $\frac{2p^{2}(2m+1)}{p(2m+1)+1}$ and $\frac{2p^{2}(2m+1)}{p(2m+1)-1}$ surgeries of $(2,p)$-torus knot ($m \in \Z$).
\end{conjecture}

The conjecture has been studied and verified for several classes of knots \cite{iis1,iis2,it1,it2,va1,va2} (see \cite{iis2} for comprehensive discussion), and there are some theoretical supporting evidences.

Nevertheless, we should remark that Conjecture \ref{conjecture:main} is a bit optimistic. There are examples of chirally cosmetic surgeries on a hyperbolic knot in a lens space \cite{bhw}, or chirally cosmetic surgeries on a hyperbolic knot in a hyperbolic 3-manifold \cite{ij}. Thus when we think about chirally cosmetic surgeries on knots in general 3-manifolds, we cannot expect an analogous conjecture.

In this paper, we study chirally cosmetic surgeries for a special alternating knot, a knot represented by a special alternating diagram. Here (up to mirror image), a special alternating diagram is an alternating diagram that is also positive (all the crossings are positive). In particular, special alternating knots are never amphicheiral. 
 
We show that, when a special alternating knot $K$ is sufficiently complicated in the sense that its twist number $tw(K)$ (see Section \ref{section:twist-number} for the definition) is large, then $K$ does not admit chirally cosmetic surgeries. 

\begin{theorem}
\label{theorem:main}
Let $K$ be a special alternating knot. If $tw(K)>63$ then $K$ does not admit chirally cosmetic surgeries.
\end{theorem}

Since a special alternating knot with $tw(K)=1$ is precisely the $(2,p)$-torus knot, our theorem says that when $K$ is far from the $(2,p)$-torus knot then it does no admit chirally cosmetic surgery. Compared with previous results, our theorem treat  more complicated and general knots so it provides a supporting evidence for Conjecture \ref{conjecture:main} having a different flavor.

We give the organization of the paper and an outline of the proof of Theorem \ref{theorem:main}.

In Section \ref{section:constraint}, we review constraints for knots to admit chirally cosmetic surgeries, based on finite type invariants $a_2(K)$, $v_3(K)$, $a_4(K)$ and determinant. 
 
In Section \ref{section:finite-type}, using Gauss diagram formulae, we review several estimates of finite type invariants.
Most of results are known, but here we provide some minor improvements or simplifications. Among them, we give the optimal upper bound of $a_2(K)$ (Proposition \ref{proposition:optimal-a_2}) and $v_3(K)$ (Proposition \ref{prop:v3-bounds}) which solve Willerton's conjectures \cite{wi}, \cite[Conjecture 2.11]{oh}.

Section \ref{section:determinant} is devoted to an estimate of the determinant of alternating knots. In \cite{st-det} Stoimenow gave an lower bound of determinant of alternating knots in terms of the twist numbers. We give a mild improvement that takes into account of crossing numbers.

In Section \ref{section:proof} we prove our main theorem by combining inequalities established so far. The estimates developed in previous sections immediately show that a special alternating knot $K$ does not admit chirally cosmetic surgeries unless a large portion of the crossings is concentrated in a single twist region. In  a remaining special case we have better estimate of finite type invariants that leads to the non-existence of chirally cosmetic surgeries.
 
Section \ref{section:alternating} provides a concise discussion for general alternating knot case. We prove similar non-existence result under some additional assumptions. We will briefly discuss a hyperbolic geometry approach that also provide a similar non-existence result of chirally cosmetic surgeries.

\section*{Acknowledgement}
The author is partially supported by JSPS KAKENHI Grant Numbers 19K03490, 21H04428.
The author wishes to express his gratitude to Kazuhiro Ichihara for many helpful and stimulating discussions.

\section{Chirally cosmetic surgery constraint}
\label{section:constraint}

In this section we quickly review sufficient conditions for alternating knots to admit no chirally cosmetic surgeries.

We denote by $a_n(K)$ the coefficient of $z^{n}$ in the Conway polynomial $\nabla_K(z)$ of a knot $K$. Let $v_3(K)$ the finite type invariant of degree three defined by 
\[ v_3(K)=-\frac{1}{144}V'''_K(1)-\frac{1}{48}V''_K(1) \in \frac{1}{4}\Z\]
where $V_K(t)$ is the Jones polynomial of $K$.
A non-standard normalization that $v_3(K) \in \frac{1}{4}\Z$ comes from an interpretation as a coefficient of Kontsevich invariant. In the following, we mainly use $4v_3(K)$ so that $4v_3(K) \in \Z$.
In a theory of finite type invariant, $a_2$ and $4v_3$ are often called the primitive finite type invariants of degree $2$ and $3$, respectively.

\begin{theorem}
\label{theorem:obstruction}
Let $K$ be an alternating knot such that $4v_3(K)\neq 0$.
\begin{itemize}
\item[(i)] \cite[Theorem 1.9]{iis2}
If 
\[ 0< \left| \frac{7a_2(K)^2-a_2(K)-10a_4(K)}{4v_3(K)}\right| \leq \frac{1}{2}\left( \det(K)-|\sigma(K)|-1\right)\]
 then $K$ has no chirally cosmetic surgeries.
\item[(ii)] \cite[Corollary 2.8]{va2} If $\sigma(K)=2g(K)$ and 
\[ \frac{7a_2(K)^2-a_2(K)-10a_4(K)}{4v_3(K)} \ne \frac{1}{2}\det(K) + 3g(K)-\frac{5}{2}\]
then $K$ has no chirally cosmetic surgeries, unless $K$ is not the $(2,p)$ torus knot.
\end{itemize}
Here $\sigma(K)$ denote the signature of $K$.
\end{theorem}

We remark that for the mirror image $\overline{K}$ of $K$, $4v_3(K)=-4v_3(\overline{K})$. Thus the assumption that $4v_3(K)\neq 0$ implies that $K$ is not amphicheiral.

\begin{remark}
The original statements are stated under the weaker assumption that $K$ is homologically thin and use the tau-invariant $\tau(K)$. For alternating knots, $\sigma(K)=2\tau(K)$ so we use more familiar and elementary invariant $\sigma(K)$.
Furthermore, \cite[Corollary 2.8]{va2} only excludes chirally cosmetic surgeries such that the slopes $r$ and $r'$ has opposite signs and $r\neq -r'$ (in \cite{iis2} we call such a chirally cosmetic surgery \emph{$-$-type}).
Our assertion follows from the \cite[Corollary 2.8]{va2}, together with the following facts.
\begin{itemize}
\item If a knot $K$ admits chirally cosmetic surgeries of the same signs of slopes (in \cite{iis2} we call such a chirally cosmetic surgery \emph{$+$-type}) then $K$ is an L-space knot \cite[Theorem 9.8]{os}.
\item An alternating knot is an L-space knot if and only if it is the $(2,p)$-torus knot
\cite[Theorem 1.5]{os1}.
\item $v_3(K)\neq 0$ implies that $r$ surgery and $-r$ surgery are not chirally cosmetic \cite{it1}.
\end{itemize} 
\end{remark}

\section{Estimate of finite type invariants from diagrams}
\label{section:finite-type}

In this section we discuss several estimates of invariants $a_2$, $a_4$ and $4v_3$ that appears in Theorem \ref{theorem:obstruction}. In particular, we provide an optimal upper bounds of $a_2$ and $4v_3$ in terms of the crossing number.
 
\subsection{Gauss diagram}

Let $(S^{1},\ast)$ be an oriented circle $S^{1}$ with a based point $\ast$.
\begin{itemize}
\item A \emph{chord} is an unordered pair $\{u,v\}$ of distinct points of $S^{1} \setminus \{\ast\}$. 
\item An \emph{(unsigned) arrow} is an ordered pair $(o,u)$ of distinct points of $S^{1} \setminus \{\ast\}$. 
\item A \emph{signed arrow} $\mathbf{a}=((o,u),\varepsilon)$ is an arrow $(o,u)$ equipped with the sign $\varepsilon \in \{\pm\}$.  
\end{itemize}

For an arrow $(o,u)$ we call $o$ and $u$ the \emph{arrow tail} and the \emph{arrow head}, respectively.

\begin{definition}
\label{definition:Gauss-diagram}
A \emph{Gauss diagram} $G$ is a set of signed arrows $\{( (o_i,u_i),\varepsilon_i)\}$ of the based circle $(S^{1},\ast)$ whose endpoints are pairwise distinct. 
%The \emph{degree} $\deg(G)$ of a Gauss diagram $G$ is the number of arrows of $G$.
A \emph{chord diagram} $C$ is defined similarly. 
\end{definition}

Two Gauss diagrams $G$ and $G'$ are the same if there exists an orientation-preserving homomorphism $f:(S^{1},\ast) \rightarrow (S^{1},\ast)$ that sends an arrow $\mathbf{a}=((o,u),\varepsilon)$ of $G$ to an arrow $f(\mathbf{a})=((f(o),f(u)),\varepsilon)$ of $G'$.

In the following, we often view a (signed/unsigned) arrow simply as a chord by forgetting additional information.
Also, by Gauss diagram we will sometime allow the case where some arrows are not signed, or, just a chord. We call such an object \emph{weak Gauss diagram}, to distinguish them from honest Gauss diagrams as defined in Definition \ref{definition:Gauss-diagram}.

As usual, we express a (weak) Gauss diagram $G$ as a diagram consisting of a circle and signed arrows by drawing an arrow from $o$ to $u$ with sign $\varepsilon$, for each signed  arrow $((o,u),\varepsilon)$ of $G$.
In the diagram expression, we always regard the circle is oriented in a counter-clockwise direction.

For a knot diagram $D$ we take a base point $\ast$. One can assign the Gauss diagram $G_D$ as follows. We view the diagram $D$ as an immersion $\gamma:S^{1} \rightarrow \R^{2}$ sending the base point $\ast$ of $S^{1}$ to the base point $\ast$ of $D$. For each double point $c$ of $D$, we assign the signed arrow $((o(c),u(c)),\varepsilon(c))$, where $o(c)$ and $u(c)$ are the preimage of the over/under arcs at the crossing $c$, and $\varepsilon(c)$ is the sign of the crossing $c$.

 For a weak Gauss diagram $A$ and a Gauss diagram $G$,  the \emph{pairing} $\langle A, G \rangle$ is defined by
\[ \langle A, G \rangle = \sum_{G' \subset G, G'=A} \varepsilon(G') \]
where the summation runs the set of all sub Gauss diagrams $G'$ of $G$ such that $G'$ is equal to $A$, and $\varepsilon(G') \in \{\pm 1\}$ is the product of the signs of all arrows of $G'$. 
Here a \emph{sub Gauss diagram} $G'$ of $G$ is a Gauss diagram whose set of signed arrows is a subset of the set of signed arrows of $G$, and `$G'$ is equal to $A$' means that after forgetting additional information of $G$ if necessary (such as, by forgetting the signs, or, by regarding an arrow as a chord) $G$ becomes the same weak Gauss diagram $A$. The pairing is naturally extended for a formal linear combination $\mathcal{C}=\sum_{i=1}a_i A_i$ of weak Gauss diagrams as
\[ \langle \mathcal{C}, G \rangle =  \sum_{i=1} a_{i}\langle A_i, G \rangle \]

For every (integer-valued) finite type invariant $v$ of knots, there exists a formal linear combination of weak Gauss diagram $\mathcal{C}$ such that
\begin{equation}
\label{eqn:Gauss-formula}
 v(K) = \langle \mathcal{C}, G_D \rangle
\end{equation}
for any diagram $D$ of a knot $K$ \cite{gpv}.
We call \eqref{eqn:Gauss-formula} a \emph{Gauss diagram formula} of the finite type invariant $v$.

\subsection{Upper bounds for the coefficient of the Conway polynomial}

Although we only need estimates of $a_2(K)$, $a_4(K)$, we give an upper bound of the coefficients $a_{2n}(K)$ of  the Conway polynomial. This is a direct generalization of Polyak-Viro's argument \cite{pv2} for estimates of $a_2(K)$.

\begin{proposition}
\label{proposition:estimate-conway}
For a knot $K$ with crossing number $c$, 
\[  |a_{2n}(K)| \leq \frac{1}{2} \left( \binom{c}{2n} - 2 \binom{\frac{c}{2}}{2n} \right)\]
holds. In particular,
\[ a_2(K) \leq \frac{c^{2}}{8}\]
\end{proposition}

We also observe the following estimate

\begin{proposition}
\label{proposition:a_2-vs-a_4}
If $D$ is a positive diagram with $c$ crossings,
\[ a_{2n}(K)\leq \frac{1}{(2n-2)!}c^{2n-2} a_2(K)\]
\end{proposition}

To prove theorem, we review a Gauss diagram formula of $a_{2n}$. A chord diagram $C$ is \emph{connected} if the curve $X_C$ obtained by taking parallel doubling at each chord 
as 
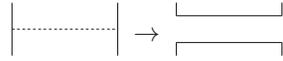
\[
\raisebox{-2mm}{
\begin{picture}(40,20)
\multiput(0,10)(2,0){20}
    {\line(1,0){1}}
\put(0,0){\line(0,1){20}}
\put(40,0){\line(0,1){20}}
\end{picture} } 
\rightarrow
\raisebox{-2mm}{
\begin{picture}(40,20)
\put(0,0){\line(0,1){5}}
\put(0,5){\line(1,0){40}}
\put(40,0){\line(0,1){5}}
\put(0,15){\line(0,1){5}}
\put(0,15){\line(1,0){40}}
\put(40,15){\line(0,1){5}}
\end{picture} } 
\]
is connected. For a connected chord diagram $C$ having $2n$ chords, we assign a weak Gauss diagram $A_C$ as follows. From the based point $\ast$, we walk along the circle $X_C$. When we first pass the portion of the circle coming from a chord $\{o,u\}$, we assign the direction according to the orientation of $X_C$.

Using the notion of connected chord diagram, the coefficient of the Conway polynomial is described as follows: 
\begin{theorem}\cite{ckr}
Let $\mathcal{C}_{2n}= \sum_{C} A_C$ be a linear combination of unsigned Gauss diagram, where the summation runs over the set of all connected chord diagrams $C$ with degree $2n$. Then for a diagram $D$ of a knot $K$,
\[ a_{2n}(K) = \langle \mathcal{C}_{2n}, G_D \rangle \]
\end{theorem}

Thus $a_{2n}(K)\geq 0$ for a positive knot $K$. Indeed, recently it is shown that for a positive knot $K$ and for $n = 1,\ldots,g(K)$ \cite{is}, a stronger inequality
\begin{equation}
\label{eqn:a2n-positive}
a_{2n}(K)\geq \binom{g(K)}{n} >0   
\end{equation}
holds.

\begin{example}[Gauss diagram formula of $a_2$]
 There are three chord diagrams 
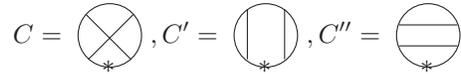
\[ C= \raisebox{-4mm}{
\begin{picture}(24,28)
\put(12,14){\circle{24}}
\put(9,0){$\ast$}
\put(21,22){\line(-1,-1){17}}
\put(3,22){\line(1,-1){17}}
\end{picture} 
}, C'=
\raisebox{-4mm}{
\begin{picture}(24,28)
\put(12,14){\circle{24}}
\put(9,0){$\ast$}
\put(19,5){\line(0,1){19}}
\put(5,5){\line(0,1){19}}
\end{picture}
},
C''=
\raisebox{-4mm}{
\begin{picture}(24,28)
\put(12,14){\circle{24}}
\put(9,0){$\ast$}
\put(1,10){\line(1,0){22}}
\put(1,18){\line(1,0){22}}
\end{picture}
}
\]
of degree two. $C$ is connected whereas $C'$ and $C''$ are not. From the connected chord diagram $C$ we get a weak Gauss diagram \GaussD so 
\begin{equation}
\label{eqn:a2}
a_{2}(K)= \langle \GaussD, G_D \rangle.
\end{equation}
\end{example}

We remark that Gauss diagram formula is far from unique. For example, using the fact that $a_2(K)$ is invariant under the mirror image, $a_2(K)$ has a different Gauss diagram formula
\begin{equation}
\label{eqn:a2-formula}
a_2(K) =  \langle
\raisebox{-4mm}{
\begin{picture}(24,28)
\put(12,14){\circle{24}}
\put(9,0){$\ast$}
\put(21,22){\vector(-1,-1){17}}
\put(3,22){\vector(1,-1){17}}
\end{picture} 
}, G_D \rangle. \end{equation}

\begin{proof}[Proof of Proposition \ref{proposition:estimate-conway}]
For an unsigned Gauss diagram $A$, we denote by $\overline{A}$ the Gauss diagram obtained by reversing all the arrows.

Let $\overline{D}$ be the mirror image of $D$. Then the Gauss diagram $G_{\overline{D}}$ of $\overline{D}$ is obtained from $G_{D}$ by reversing all the arrows and signs.
Thus
\[a_{2n}(K)= a_{2n}(\overline{K}) = \langle \mathcal{C}_{2n},G_{\overline{D}} \rangle =  \langle \overline{\mathcal{C}_{2n}}, G_{D} \rangle \]
Here $\overline{\mathcal{C}_{2n}} = \sum_{C} \overline{A_{C}}$.
Thus we conclude
\[ a_{2n}(K)=\frac{1}{2} \langle  \mathcal{C}_{2n} + \overline{\mathcal{C}_{2n}},G_{D} \rangle. \]

We say that an arrow of a Gauss diagram is a \emph{left arrow} (resp. a \emph{right arrow}) if it is of the form $\Larrow$ (resp. $\Rarrow$).
We denote by  $c_{\leftarrow}$ and $c_{\rightarrow}$ the number of left arrow and right arrows of $G_D$, respectively.

By the construction, the unsigned arrow diagram $A_C$ from a connected chord diagram $C$ has the following properties:
\begin{itemize}
\item[(i)] The endpoints of two arrows adjacent to the base point $\ast$ are arrow tails. 
\item[(ii)] $A_C \neq A_{C'}$ if two chord diagrams $C$, $C'$ are different. 
\item[(iii)] $A_C$ contains at least one left arrows and at least one right arrow. 
\end{itemize}

The property (i) says that $A_C \neq \overline{A_{C'}}$ for any connected chord diagrams $C$ and $C'$. Thus for each sub Gauss diagram $G'$ of $G_D$ having $2n$ chords,
$|\langle \mathcal{C}_{2n}+\overline{\mathcal{C}_{2n}}, G' \rangle| \leq 1$.
Moreover, the property (iii) says that  $|\langle \mathcal{C}_{2n}, G' \rangle| =0$ if all the arrows of $G'$ are left arrows or right arrows.

The number of sub Gauss diagrams of degree $2n$ having at least one left arrow and one right arrow is $\binom{c}{2n} -\binom{c_{\leftarrow}}{2n} - \binom{c_{\rightarrow}}{2n}$.
Therefore
\begin{align*}
a_{2n}(K) &\leq \frac{1}{2}\left(\binom{c}{2n} -\binom{c_{\leftarrow}}{2n} - \binom{c_{\rightarrow}}{2n} \right) \\
& \leq \frac{1}{2} \left( \binom{c}{2n} - 2 \binom{\frac{c}{2}}{2n} \right) 
\end{align*}

\end{proof}

In this opportunity we point out the following best-possible estimate\footnote{This was implicit in \cite[Remark 1.G]{pv2}, though they did not state explicitly.} of $a_2(K)$, which was conjectured by Willerton \cite{wi} (we remark that for the $(2,p)$-torus knot $K$ $a_2(K)=\frac{c^{2}-1}{8})$.

\begin{proposition}[Optimal bound of $a_2(K)$]
\label{proposition:optimal-a_2}
For a non-trivial knot $K$,
\[ a_2(K) \leq \frac{c(K)^{2}-1}{8} \]
\end{proposition}
\begin{proof}
This is done by a slightly careful re-examination of the proof of Proposition \ref{proposition:estimate-conway}.
Let $D$ be a diagram of $K$ with $c=c(K)$ crossings.

First assume that $c(K)$ is odd. Although this case the assertion is obvious because $a_2(K)$ is integer, here we give a slightly different argument that can be used for general cases. Since $c_{\leftarrow}$ and $c_{\rightarrow}$ are integers 
\begin{align*}
a_{2}(K) &
\leq \frac{1}{2}\left(\binom{c}{2} -\binom{c_{\leftarrow}}{2} - \binom{c_{\rightarrow}}{2} \right) \\
& \leq \frac{1}{2} \left(  \binom{c}{2} - \binom{\frac{c+1}{2}}{2} - \binom{\frac{c-1}{2}}{2} \right) = \frac{c^{2}-1}{8}.
\end{align*}
Next we assume that $c(K)$ is even. Recall that for a Gauss diagram $G_D$ from knot diagram, for each arrow $\mathbf{a}$, the number of arrows that intersects with $\mathbf{a}$ is even. Thus when $c(K)$ is even, the number of sub (weak) Gauss diagrams of $G_D$ of the form $\raisebox{-4mm}{
\begin{picture}(24,28)
\put(12,14){\circle{24}}
\put(9,0){$\ast$}
\put(21,22){\line(-1,-1){17}}
\put(3,22){\line(1,-1){17}}
\end{picture} 
}$ is at most $\frac{c(c-2)}{2}$. Therefore we conclude 
\begin{align*}
a_{2}(K) &
\leq \frac{1}{2}\left(\frac{c(c-2)}{2} -\binom{c_{\leftarrow}}{2} - \binom{c_{\rightarrow}}{2} \right) \\
& \leq \frac{1}{2} \left( \frac{c(c-2)}{2} - 2 \binom{\frac{c}{2}}{2} \right) = \frac{c^{2}-2c}{8} < \frac{c^{2}-1}{8}
\end{align*}
\end{proof}

\begin{remark}
A similar argument can be used to give a slight improvement of Proposition \ref{proposition:estimate-conway}, but it looks to be far from optimal. Although exploring the optimal upper bound of $a_4$ or $a_{2n}$ is of independent interest (see \cite[Problem 1.17]{oh}), we do not pursue this.
\end{remark}

\begin{proof}[Proof of Proposition \ref{proposition:a_2-vs-a_4}]

For a connected Gauss diagram $C$ and its unsigned Gauss diagram $A_C$, two arrows adjacent to the base point forms a sub Gauss diagram of the form either $\GaussD$ or $\raisebox{-4mm}{
\begin{picture}(24,28)
\put(12,14){\circle{24}}
\put(9,0){$\ast$}
\put(4,5){\vector(0,1){17}}
\put(20,5){\vector(0,1){17}}
\end{picture} 
}$. For the latter case, since $C$ is connected, there must exists a left arrow that intersects with the right arrow $\raisebox{-4mm}{
\begin{picture}(24,28)
\put(12,14){\circle{24}}
\put(9,0){$\ast$}
\put(4,5){\vector(0,1){17}}
\end{picture}}$. 
Thus in both cases, $A_C$ contains $\GaussD$ as its sub Gauss diagram.
 Since the number of degree two sub Gauss diagrams of $G_D$ which are equal to $\GaussD$ is $a_2(K)$, we conclude 
\[ a_{2n}(K) \leq \binom{c-2}{2n-2}\langle \GaussD, G_D \rangle  \leq  \frac{1}{(2n-2)!}c^{2n-2} a_2(K).\] 
\end{proof}

\subsection{Bounds for $4v_3$}

We review some estimates of $4v_3(K)$. Most results are taken from \cite{st}\footnote{We remark that the invariant $v_3$ in Stoimenow's paper \cite{st} is equal to $16v_3(K)$ in our paper.}, but here we make arguments simpler or to make slight improvements, by using the following Gauss diagram formula.
\begin{proposition}\cite[Page 493, last line]{cp}
\[ 4v_3(K) = \langle \mathcal{C}_{A} +\mathcal{C}_{B}, G_D \rangle \]
where
\begin{align*}
\mathcal{C}_A & = 
\raisebox{-4mm}{
\begin{picture}(24,28)
\put(12,14){\circle{24}}
\put(9,0){$\ast$}
\put(21,22){\vector(-1,-1){17}}
\put(3,22){\vector(1,-1){17}}
\put(23,17){\line(-1,0){23}}
\end{picture} 
}
+
\raisebox{-4mm}{
\begin{picture}(24,28)
\put(12,14){\circle{24}}
\put(9,0){$\ast$}
\put(21,22){\vector(-1,-1){17}}
\put(3,22){\vector(1,-1){17}}
\put(10,26){\vector(1,-1){14}}
\end{picture} 
}
+
\raisebox{-4mm}{
\begin{picture}(24,28)
\put(12,14){\circle{24}}
\put(9,0){$\ast$}
\put(21,22){\vector(-1,-1){17}}
\put(3,22){\vector(1,-1){17}}
\put(14,26){\vector(-1,-1){14}}
\end{picture} 
}
+
\raisebox{-4mm}{
\begin{picture}(24,28)
\put(12,14){\circle{24}}
\put(9,0){$\ast$}
\put(19,24){\vector(-1,-1){17}}
\put(5,24){\vector(1,-1){17}}
\qbezier(23,18)(14,14)(16,3)
\put(16,3){\vector(0,-1){0}}
\end{picture} 
},\\
\mathcal{C}_B & =
\raisebox{-4mm}{
\begin{picture}(24,28)
\put(12,14){\circle{24}}
\put(9,0){$\ast$}
\put(22,20){\scriptsize $+$}
\put(-5,20){\scriptsize  $+$}
\put(21,22){\vector(-1,-1){17}}
\put(3,22){\vector(1,-1){17}}
\end{picture} 
}
-
\raisebox{-4mm}{
\begin{picture}(24,28)
\put(12,14){\circle{24}}
\put(9,0){$\ast$}
\put(22,20){\scriptsize  $-$}
\put(-5,20){\scriptsize  $-$}
\put(21,22){\vector(-1,-1){17}}
\put(3,22){\vector(1,-1){17}}
\end{picture}.
}
\end{align*}
\end{proposition}

This formula leads to an slight improvement of \cite[Proposition 7.2]{st}.
\begin{lemma}
\label{lemma:stoimenow-improvement}
Let $K$ be a knot represented by a positive diagram $D$, 
\[ a_2(K) \leq 4v_3(K)\]
\end{lemma}
\begin{proof}
For a positive diagram $D$, $\langle \mathcal{C}_A, G_D \rangle \geq 0$ and $\langle \mathcal{C}_B, G_D \rangle =a_2(K)$ by \eqref{eqn:a2-formula}. Therefore $4v_3(K) \geq a_2(K)$. 
\end{proof}

In particular, by \eqref{eqn:a2n-positive}, for a non-trivial positive knot $K$
\begin{equation}
\label{eqn:v3-positive} 0 < g(K) \leq a_2(K) \leq 4v_3(K).
\end{equation}
(We remark that $a_2(K)\geq g(K)$ and $4v_3(K)\geq g(K)$ were first proven in \cite[Theorem 6.2]{st} and \cite[Theorem 5.1]{st}, respectively).

We point out the following optimal estimate of $4v_3(K)$ conjectured by Willerton \cite{wi}, \cite[Conjecture 2.11]{oh} (we remark that for the $(2,p)$-torus knot $K$, $4v_3(K)=\textsf{sgn}(p)\frac{c(K)^{3}-c(K)}{24})$.

\begin{proposition}[Optimal bound of $4v_3(K)$]
\label{prop:v3-bounds}
\[ -\frac{1}{24}(c(K)^{3}-c(K)) \leq 4v_3(K) \leq \frac{1}{24}(c(K)^{3}-c(K))\]
\end{proposition}
\begin{proof}
For a diagram $D$ let $D_+$ be the positive diagram obtained from $D$ by suitably changing over-under information at each crossings. It is known that $a_2(D)\leq a_2(D_+)$ and $v_3(D)\leq v_3(D_+)$ \cite[Theorem 5.2]{st}. 
Let $D$ be a minimum crossing diagram of a knot $K$. By \cite[Theorem 7.2]{st},  $4v_3(K) \leq \frac{c(D_+)}{3}a_2(K)$. Therefore by Proposition \ref{proposition:optimal-a_2} we conclude that 
\[ 4v_3(K) \leq 4v_3(D_+) \leq \frac{c(K)}{3}a_2(D_+) \leq \frac{c(K)^3-c(K)}{24}. \]
The lower bound follows from the property $4v_3(\overline{K})=-4v_3(K)$.
\end{proof}

\section{Determinant estimate}
\label{section:determinant}

In this section we develop a lower bound of the determinant of alternating knots.

\subsection{Twist regions and twist number}
\label{section:twist-number}

A \emph{twist region} $R$ of a knot diagram $D$ is a maximum non-empty sub-diagram that consists of non-trivial twists of two parallel strands as 
\[ 
\begin{picture}(90,30)
\put(0,23){\line(1,0){20}}
\put(0,7){\line(1,0){20}}
\put(20,0){\framebox(30,30){}}
\put(30,10){\LARGE $R$}
\put(50,23){\line(1,0){16}}
\put(50,7){\line(1,0){16}}
\put(75,10){$=$}
\end{picture}
\begin{picture}(120,30)
\put(0,3){\line(1,1){24}}
\put(10,17){\line(-1,1){10}}
\put(14,13){\line(1,-1){10}}
\put(24,27){\line(1,-1){10}}
\put(24,3){\line(1,1){24}}
\put(38,13){\line(1,-1){10}}
\put(52,13){$\cdots$} 
\put(68,3){\line(1,1){24}}
\put(68,27){\line(1,-1){10}}
\put(82,13){\line(1,-1){10}}
\put(102,15){\mbox{or}}
\end{picture} 
\begin{picture}(100,30)
\put(0,3){\line(1,1){10}}
\put(14,17){\line(1,1){10}}
\put(0,27){\line(1,-1){24}}
\put(24,27){\line(1,-1){24}}
\put(24,3){\line(1,1){10}}
\put(38,17){\line(1,1){10}}
%\put(48,0){\line(1,1){24}}
%\put(48,24){\line(1,-1){10}}
%\put(62,10){\line(1,-1){10}}
\put(52,13){$\cdots$} 
\put(68,27){\line(1,-1){24}}
\put(68,3){\line(1,1){10}}
\put(82,17){\line(1,1){10}}
\end{picture} \]
Here maximum means that $R$ contains as many crossings as possible.
We denote by $c(R)$ the number of crossings in the twist region $R$.
Two twist regions $R$ and $R'$ are \emph{equivalent} if
there exists a circle $C$ in the projection plane $\R^{2}$ such that $C$ is disjoint from $D$, except $C$ traverses $R$ and $R'$ once as in 
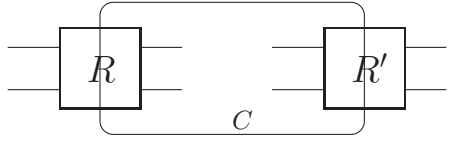
\[
\begin{picture}(200,50)
\put(0,33){\line(1,0){20}}
\put(0,17){\line(1,0){20}}
\put(20,10){\framebox(30,30){}}
\put(30,20){\LARGE $R$}
\put(50,33){\line(1,0){16}}
\put(50,17){\line(1,0){16}}
\put(100,33){\line(1,0){20}}
\put(100,17){\line(1,0){20}}
\put(120,10){\framebox(30,30){}}
\put(130,20){\LARGE $R'$}
\put(150,33){\line(1,0){16}}
\put(150,17){\line(1,0){16}}
\put(85,25){\oval[5](100,50)}
\put(85,2){$C$}
\end{picture} 
\]
We call such a circle $C$ an \emph{equivalence circle} between $R$ and $R'$.

The number of twist regions of the diagram $D$ is called the \emph{twist number} and denoted by $tw(D)$. We denote by $\overline{tw}(D)$ the number of equivalence classes of twist regions of a diagram $D$.
% and we call $\overline{tw}(D)$ the \emph{reduced twist number} of $D$.

\begin{definition}
A diagram $D$ is \emph{twist-reduced} if no two distinct twist regions are equivalent. 
\end{definition}

Thus for a twist-reduced diagram $D$, $tw(D)=\overline{tw}(D)$. When two twist regions $R$ and $R'$ are equivalent then their equivalence circle shows that $R$ and $R'$ are consolidated into a single twist region by flype 
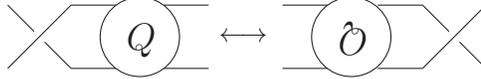
\[
\begin{picture}(200,30)
\put(0,3){\line(1,1){24}}
\put(10,17){\line(-1,1){10}}
\put(14,13){\line(1,-1){10}}
\put(50,15){\circle{30}}
\put(45,10){\LARGE $Q$}
\put(24,27){\line(1,0){17}}
\put(24,3){\line(1,0){17}}
\put(59,27){\line(1,0){17}}
\put(59,3){\line(1,0){17}}
\put(80,13){$\longleftrightarrow$}
\put(104,27){\line(1,0){17}}
\put(104,3){\line(1,0){17}}
\put(130,15){\circle{30}}
\put(125,18){\LARGE \rotatebox{180}{$Q$}}
\put(140,27){\line(1,0){17}}
\put(140,3){\line(1,0){17}}
\put(157,3){\line(1,1){24}}
\put(167,17){\line(-1,1){10}}
\put(171,13){\line(1,-1){10}}
\end{picture} 
\]
(and Redemeister move II, if necessary).
In particular, every alternating diagram $D$ can be made twist-reduced by applying flypes.  

\begin{definition}
The \emph{twist number} of an alternating knot $K$ is defined by
\begin{align*}
 tw(K) & = \min\{\overline{tw}(D) \: | \: D \mbox{ is an alternating diagram of }K\}\\
 &= \min\{tw(D) \: | \: D \mbox{ is a reduced, twist-reduced alternating diagram of }K\}.
\end{align*}
\end{definition}

\subsection{Twist number, crossing number and determinant}

In \cite[Theorem 4.3]{st-det} Stoimenow showed that for a reduced alternating diagram $D$ of a link $L$, the inequality
\begin{equation}
\label{eqn:st-det}
\det(L) \geq 2\gamma^{\btw(D)-1} 
\end{equation}
holds. Here $\gamma=1.425...$ is the inverse of the positive root of $x^{5}+2x^{4}+x^{3}-1=0$. In practice, since $\gamma \approx \sqrt{2}=1.412...$ one can use $\sqrt{2}$ as an approximation of the constant $\gamma$. 

By taking into account of the crossing numbers we get an improvement of \eqref{eqn:st-det}.

\begin{theorem}
\label{theorem:det-key}
Let $D$ be a reduced alternating diagram of a link $L$ with $\btw(D)\geq 2$. 
Then
\[ \det(L) > 2\gamma^{-1}\left(\gamma^{\btw(D)} + (c(D)-\btw(D))\gamma^{(\btw(D)-1)/2} \right). \]
Moreover, if $D$ contains a twist region $R$ that contains $c(R)>2$ crossings, then
\begin{align*}
\det(L) &> 2\gamma^{-1}\Bigl(\gamma^{\btw(D)} + (c(D)-\btw(D))\gamma^{(\btw(D)-1)/2}\\ & \qquad \qquad + (c(R)-2)(c(D)-\btw(D)-c(R))\gamma^{(\btw(D)-3)/4}\Bigr)
\end{align*}
\end{theorem}

We remark that assertion is not true when $\btw(D)=1$, namely. $D$ is the $(2,p)$-torus knot/link diagrams.

\begin{proof}
With no loss of generality we may assume that $D$ is twist-reduced so $\btw(D)=tw(D)$. 

By direct computations one can check that the assertion holds for the case $tw(D)=2$, the double twist knot/link diagram case, or the 3-strand pretzel knot/link diagram case. In the following we assume that $tw(D)>2$ and $D$ is not the 3-strand pretzel knot/link diagrams.

We prove the assertion by induction on $((c(D)-tw(D),c(D))$.
If $c(D)-tw(D)=0$, then $R$ contains exactly one crossing so $c(R)=1$. Thus in this case the desired inequality is nothing but \eqref{eqn:st-det}.

Assume that $c(D)-tw(D)>0$.
By taking the mirror image if necessary, we may assume that the twist region $R$ consists of the crossings of the form 
$\raisebox{-3mm}{
\begin{picture}(24,24)
\put(0,0){\line(1,1){24}}
\put(14,10){\line(1,-1){10}}
\put(0,24){\line(1,-1){10}}
\end{picture} }$.

Take a crossing $c$ of $R$ and let $D_{0}$ be the diagram obtained by resolving the crossing $c$ as 
$ 
\raisebox{-3mm}{
\begin{picture}(24,24)
\put(0,0){\line(1,1){24}}
\put(14,10){\line(1,-1){10}}
\put(0,24){\line(1,-1){10}}
\end{picture} }
 \longrightarrow \raisebox{-3mm}{
\begin{picture}(24,24)
\qbezier(0,0)(12,14)(24,0)
\qbezier(0,24)(12,10)(24,24)
\end{picture}}$. Similarly, let $D_{\infty}$ be the diagram obtained by resolving the crossing $c$ as
$ \raisebox{-3mm}{
\begin{picture}(24,24)
\put(0,0){\line(1,1){24}}
\put(14,10){\line(1,-1){10}}
\put(0,24){\line(1,-1){10}}
\end{picture} }
 \longrightarrow \raisebox{-3mm}{
\begin{picture}(24,24)
\qbezier(0,0)(14,12)(0,24)
\qbezier(24,0)(10,12)(24,24)
\end{picture} }$ and removing the trivial kinks in $R$ by Reidemeister move I.
Thus $D_{\infty}$ is the diagram obtained by resolving the twist region $R$ as
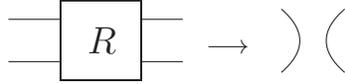
\[ 
\begin{picture}(90,30)
\put(0,23){\line(1,0){20}}
\put(0,7){\line(1,0){20}}
\put(20,0){\framebox(30,30){}}
\put(30,10){\LARGE $R$}
\put(50,23){\line(1,0){16}}
\put(50,7){\line(1,0){16}}
\put(75,10){$\longrightarrow$}
\end{picture} \quad \raisebox{1mm}{
\begin{picture}(24,24)
\qbezier(0,0)(14,12)(0,24)
\qbezier(24,0)(10,12)(24,24)
\end{picture} } \]

The diagram $D_{0}$ is reduced. The diagram $D_{\infty}$ is also reduced because of a crossing $c$ of $D_{\infty}$ is nugatory then the crossing $c$ in $D$ must be contained in the twist region in $R$. As for the crossing numbers we have $c(D_0)=c(D)-1$ and $c(D_\infty)=c(D)-c(R)$. \\

\textbf{Case 1: $D_0$ is twist-reduced}\\

In this case $\overline{tw}(D_0)=tw(D_0)=\overline{tw}(D)-1$.
In the diagram $D_0$, the twist region $R$ contains $c(R)-1$ crossings so by induction
\begin{align*}
\det(D_0) & > 2\gamma^{-1} \Bigl(\gamma^{\btw(D)} + ((c(D)-1)-\btw(D))\gamma^{(\btw(D)-1)/2} \\
& \qquad \qquad + (c(R)-3)(c(D)-c(R)-\btw(D))\gamma^{(\btw(D)-3)/4}) \Bigr) 
\end{align*}

To understand $\btw(D_\infty)$, we observe the following (we remark that this claim does not use the assumption that $D_0$ is twist-reduced, and the claim will be used to the Case 2 below).

\begin{claim}
\label{claim:1}
For each twist region $R_0$ of $D_{\infty}$ there is at most one twist region $R_1(\neq R_0)$ of $D_{\infty}$ which is twist equivalent to $R_0$.
\end{claim}
\begin{proof}[Proof of Claim \ref{claim:1}]
Assume to the contrary that there are more than one twist regions $R_1$, $R_2$ which is twist equivalent to $R_0$, in $D_{\infty}$. Let $C_i$ ($i=1,2$) be equivalence circles between $R_0$ and $R_i$. Then by taking an appropriate connected sum of $C_1$ and $C_2$ yields an equivalence circle $C$ between $R_1$ and $R_2$ in $D$ as
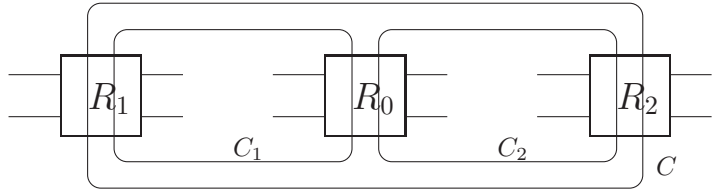
\[
\begin{picture}(300,80)
\put(0,43){\line(1,0){20}}
\put(0,27){\line(1,0){20}}
\put(20,20){\framebox(30,30){}}
\put(30,30){\LARGE $R_1$}
\put(50,43){\line(1,0){16}}
\put(50,27){\line(1,0){16}}
\put(100,43){\line(1,0){20}}
\put(100,27){\line(1,0){20}}
\put(120,20){\framebox(30,30){}}
\put(130,30){\LARGE $R_0$}
\put(150,43){\line(1,0){16}}
\put(150,27){\line(1,0){16}}
\put(85,35){\oval[5](90,50)}
\put(85,12){$C_1$}
\put(200,43){\line(1,0){20}}
\put(200,27){\line(1,0){20}}
\put(220,20){\framebox(30,30){}}
\put(230,30){\LARGE $R_2$}
\put(250,43){\line(1,0){16}}
\put(250,27){\line(1,0){16}}
\put(185,35){\oval[5](90,50)}
\put(185,12){$C_2$}
\put(135,35){\oval[5](210,70)}
\put(245,5){$C$}
\end{picture} 
\]
This contradicts the assumption that $D$ is twist reduced.
\end{proof}
By Claim \ref{claim:1}
\[ \btw(D_\infty) > \frac{tw(D_\infty)}{2} = \frac{\btw(D)-1}{2}.\]
Since we have assumed that $tw(D)=\btw(D)>2$, $\btw(D_{\infty})>1$.
Therefore by induction
\[ \det(L_\infty) \geq 2\gamma^{-1}\left( \gamma^{(\btw(D)-1)/2} + ((c(D)-c(R))-\btw(D))\gamma^{(\btw(D)-3)/4} \right) \]
Therefore
\begin{align*}
\det(L) &= \det(D_0) + \det(D_\infty)\\
&> 2\gamma^{-1} \Bigl(\gamma^{\btw(D)} + (c(D)-\btw(D))\gamma^{(\btw(D)-1)/2} \\
& \qquad \qquad + (c(R)-2)(c-\btw(D)-c(R))\gamma^{(\btw(D)-1)/4}\Bigr).\\
\end{align*}

\textbf{Case 2: $D_0$ is not twist-reduced} \\

This can happen only if $c(R)=2$. In the diagram $D_0$, the twist region $R$ consists of a single crossing that is twist equivalent to the other twist region $R'$, by changing how to view the direction of twisting as  
\[
\begin{picture}(150,130)
\put(0,120){\line(1,0){10}}
\put(0,100){\line(1,0){10}}
\put(10,120){\line(1,-1){5}}
\put(30,100){\line(-1,1){5}}
\put(10,100){\line(1,1){20}}
\put(30,100){\line(1,1){20}}
\put(30,120){\line(1,-1){5}}
\put(50,100){\line(-1,1){5}}
\put(10,95){\framebox(40,30){}}
\put(23,105){\LARGE $R$}
\put(50,120){\line(1,0){10}}
\put(50,100){\line(1,0){10}}
\put(20,60){\line(0,1){10}}
\put(40,60){\line(0,1){10}}
\put(15,30){\framebox(30,30){}}
\put(20,20){\line(0,1){10}}
\put(40,20){\line(0,1){10}}
\put(25,50){\rotatebox{270}{\LARGE $R'$}}
\put(0,0){Diagram $D$}
\put(100,120){\line(1,0){15}}
\put(100,100){\line(1,0){15}}
\put(115,100){\line(1,1){20}}
\put(115,120){\line(1,-1){5}}
\put(135,100){\line(-1,1){5}}
\put(110,95){\framebox(30,30){}}
\put(120,115){\rotatebox{270}{\LARGE $R$}}
\put(135,120){\line(1,0){15}}
\put(135,100){\line(1,0){15}}
\put(115,60){\line(0,1){10}}
\put(135,60){\line(0,1){10}}
\put(110,30){\framebox(30,30){}}
\put(115,20){\line(0,1){10}}
\put(135,20){\line(0,1){10}}
\put(120,50){\rotatebox{270}{\LARGE $R'$}}
\put(95,0){Diagram $D_0$}
\put(125,75){\oval[10](80,70)}
\put(155,30){$C_0$}
\end{picture} 
\]
Let $C_0$ be the equivalence circle between $R$ and $R'$ in $D_{0}$.

In this case, $\overline{tw}(D_0)= tw(D)-1$ because when two twist regions $R_0$ and $R_1$ in $D_0$ are equivalent and $(R_0,R_1) \neq (R,R')$, then their equivalence circle $C$ yields and equivalence circle between corresponding twist regions in $D$. In particular, $\btw(D_0)>2$ since we are assuming $\btw(D)>2$.\\

In this case, we observe that unlike Case 1, $\overline{tw}(D_{\infty})$ can decrease by at most two.
\begin{claim}
\label{claim:2}
$\overline{tw}(D_{\infty})\geq tw(D)-2$. Moreover, when $D$ is not the 3-strand pretzel knot/link diagram, then $\btw(D_{\infty})>1$.
\end{claim}
\begin{proof}[Proof of Claim \ref{claim:2}] 
Assume that two twist regions $R_0$ and $R_1$ in $D_{\infty}$ are equivalent and let $C$ be its equivalence circle. Since $R_0$ and $R_1$ are not equivalent in $D$, $C$ must traverse the twist region $R$.
 Since the twist region $R$ and $R'$ become equivalent in $D_0$, this implies that either $R_0=R'$ or $R_1 = R'$ because otherwise in the diagram $D_{\infty}$ two equivalence circles $C$ and $C_0$ intersects in the twist region $R$ hence they form a configuration 
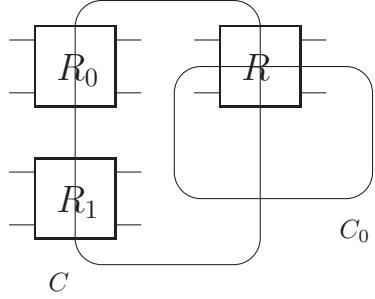
\[
\begin{picture}(180,100)
\put(0,90){\line(1,0){10}}
\put(0,70){\line(1,0){10}}
\put(10,65){\framebox(30,30){}}
\put(40,90){\line(1,0){10}}
\put(40,70){\line(1,0){10}}
\put(18,75){\LARGE $R_0$}
\put(0,40){\line(1,0){10}}
\put(0,20){\line(1,0){10}}
\put(10,15){\framebox(30,30){}}
\put(40,40){\line(1,0){10}}
\put(40,20){\line(1,0){10}}
\put(18,25){\LARGE $R_1$}
\put(70,90){\line(1,0){10}}
\put(70,70){\line(1,0){10}}
\put(80,65){\framebox(30,30){}}
\put(110,90){\line(1,0){10}}
\put(110,70){\line(1,0){10}}
\put(88,75){\LARGE $R$}
\put(60,55){\oval[10](70,100)}
\put(100,55){\oval[10](75,50)}
\put(15,-5){$C$}
\put(125,15){$C_0$}
%

%\put(100,140){\line(1,0){15}}
%\put(100,120){\line(1,0){15}}
%\put(115,120){\line(1,1){20}}
%\put(115,140){\line(1,-1){5}}
%\put(135,120){\line(-1,1){5}}
%\put(110,115){\framebox(30,30){}}
%\put(120,135){\rotatebox{270}{\LARGE $R$}}
%\put(135,140){\line(1,0){15}}
%\put(135,120){\line(1,0){15}}
%\put(115,80){\line(0,1){10}}
%\put(135,80){\line(0,1){10}}
%\put(110,50){\framebox(30,30){}}
%\put(115,40){\line(0,1){10}}
%\put(135,40){\line(0,1){10}}
%\put(120,70){\rotatebox{270}{\LARGE $R'$}}
%\put(95,20){Diagram $D_0$}
%\put(125,95){\oval[5](80,70)}
%\put(155,50){$C_0$}
\end{picture} 
\]
This is impossible.

By Claim \ref{claim:1}, this implies that there are at most one pair of twist regions of $D_{\infty}$ that is twist equivalent. In particular, if $\btw(D_{\infty})=1$ happens, then it means that $D$ must be the 3-strand pretzel knot/link diagram.
\end{proof}

When $tw(D_{\infty})=tw(D)-2$ by induction
\begin{align*}
\det(L) &= \det(D_0) + \det(D_\infty)\\
&>  2\gamma^{-1}\left(\gamma^{tw(D)-1} + (c(D)-tw(D))\gamma^{(tw(D)-2)/2} \right)\\
& \qquad +  2\gamma^{-1}\left(\gamma^{tw(D)-2} + (c(D)-tw(D))\gamma^{(tw(D)-3)/2} \right)\\
& = 2\gamma^{-1}\left(\gamma^{tw(D)}(\gamma^{-1}+\gamma^{-2}) + (\gamma^{-1/2}+\gamma^{-1})(c(D)-tw(D)) \gamma^{(tw(D)-1)/2} \right)
\end{align*}
Since $(\gamma^{-1}+\gamma^{-2})>1$ and $\gamma^{-1/2}+\gamma^{-1}>1$, we get
\[ \det(L) >  2\gamma^{-1}\left(\gamma^{tw(D)}+ (c-tw(D))\gamma^{(tw(D)-1)/2} \right). \]
Similarly, when $tw(D_{\infty})=tw(D)-1$ by induction $tw(D_{\infty})=tw(D)-1$
\begin{align*}
\det(L) &= \det(D_0) + \det(D_\infty)\\
&>  2\gamma^{-1}\left(\gamma^{tw(D)-1} + (c(D)-tw(D))\gamma^{(tw(D)-2)/2} \right)\\
& +  2\gamma^{-1}\left(\gamma^{tw(D)-1} + (c(D)-1-tw(D))\gamma^{(tw(D)-2)/2} \right)\\
& = 2\gamma^{-1}\left( 2\gamma^{tw(D)-1}+(2\gamma^{-1/2}(c(D)-tw(D))-\gamma^{-1/2})\gamma^{(tw(D)-1)/2} \right)
\end{align*}If $c(D)-tw(D)>1$ then 
\[ (2\gamma^{-1/2}(c(D)-tw(D))-\gamma^{-1/2}) > 3\gamma^{-1/2}>1 \]
so we conclude 
\[ \det(L) > 2\gamma^{-1}\left(\gamma^{tw(D)}+ (c(D)-tw(D))\gamma^{(tw(D)-1)
)/2} \right). \]
When $c(D)-tw(D)=1$, then
\begin{align*}
\det(L) &= 2\gamma^{-1}\left( 2\gamma^{tw(D)-1} + \gamma^{(tw(D)-2)/2} \right)\\
&=  2\gamma^{-1}\left( \gamma^{tw(D)} + (2-\gamma)\gamma^{tw(D)-1} +\gamma^{-1/2}\gamma^{(tw(D)-1)/2} \right)\\
&= 2\gamma^{-1}\left( \gamma^{tw(D)} + ((2-\gamma)\gamma^{(tw(D)-1)/2}+ \gamma^{-1/2})\gamma^{(tw(D)-1)/2} \right)\\
&>  2\gamma^{-1}\left( \gamma^{tw(D)} + \gamma^{(tw(D)-1)/2} \right).
\end{align*}
\end{proof}

To represent to what extent a twist region contains crossings we introduce the following quantity.

\begin{definition}
The \emph{density} $d(D)$ of diagram $D$ by
\[ d(D)=\max\left\{ \frac{c(R)}{c(D)} \: \middle| \: R \mbox{ is a twist region of }D \right\}. \]
The \emph{maximum twist region} $R$ is the twist region $R$ that attains the density, namely, $c(R)$ is the largest among the twist regions of $D$.
\end{definition}

We will use the following obvious estimate of the density in terms of the crossing number and twist numbers.

\begin{lemma}
\label{lemma:tw-density}
For a reduced diagram $D$, 
\[ \frac{1}{tw(D)} \leq d(D), \mbox { and, } \ \frac{tw(D)-1}{c(D)} \leq 1-d(D) \]
\end{lemma}
\begin{proof}
Let $R$ be the maximum twist region.
Since every twist region contains at least one crossing, we have $\frac{c(D)}{tw(D)} \leq c(R) =d(D)c(D)$ and $c(R)+ (tw(D)-1) \leq c(D)$.
\end{proof}

Using the density, we get the following estimate which is quadratic with respect to $c(D)$.

\begin{proposition}
\label{prop:det-density}
Let $D$ be a reduced diagram of a knot $K$. If $\btw(D)>1$ then
\[ \det(K) \geq 2\frac{(1-d(D))\gamma^{(\btw(D)-7)/4}}{\btw(D)}c(D)^2 \]
\end{proposition}
\begin{proof}
Let $R$ be the maximum twist region. Then 
\begin{align*}
\det(K) &> 2\gamma^{-1}\Biggl(\gamma^{\btw(D)} + (c(D)-\btw(D))\gamma^{(\btw(D)-1)/2}\\ & \qquad \qquad + (c(R)-2)(c(D)-\btw(D)-c(R))\gamma^{(\btw(D)-3)/4}\Biggr) \\
& > 2\gamma^{-1}\Biggl(\gamma^{\btw(D)} + (c(D)-\btw(D))\gamma^{(\btw(D)-1)/2}\\ & \qquad \qquad + \Bigl(\frac{c(D)}{tw(D)}-2 \Bigr)\Bigl( (1-d(D))c(D)-tw(D) \Bigr)\gamma^{(\btw(D)-3)/4}\Biggr)\\
& > 2\gamma^{-1}\Biggl(\gamma^{\btw(D)} + (c(D)-\btw(D))\gamma^{(\btw(D)-1)/2}\\ & \qquad \qquad + \Bigl(\frac{1-d(D)}{tw(D)}c(D)^2+ (-3+2d(D))c(D) + 2tw(D) \Bigr)\gamma^{(\btw(D)-3)/4}\Biggr)\\
& \geq 2\frac{(1-d(D))\gamma^{(\btw(D)-7)/4}}{tw(D)}c(D)^2 
\end{align*}
\end{proof}

\section{Proof of Theorem \ref{theorem:main}}
\label{section:proof}

In this section we prove Theorem \ref{theorem:main}.
Since $\sigma(K)=2g(K)$ holds for a positive knot $K$, we may use Theorem \ref{theorem:obstruction} (ii) to show non-existence of chirally cosmetic surgeries. 

\begin{remark}
Although we use Theorem \ref{theorem:obstruction} (ii), with a bit additional effort one can use Theorem \ref{theorem:obstruction} (i) instead, because we essentially uses an estimate of $\det(K)$, and the absolute value of the right-hand side in Theorem \ref{theorem:obstruction} (i) can be removed when $K$ is a positive knot, as we have mentioned in \cite[Remark 1.11]{iis2}.
\end{remark}

As a warm-up, we observe the non-existence of chiraly cosmetic surgery when the density is not close to $1$.

\begin{proposition}
\label{prop:non-large-density}
Let $D$ be a reduced, twist-reduced special alternating diagram.
If $d(D)\leq 1-\frac{7tw(D)}{8}\gamma^{(7-tw(D))/4}$ then $K$ does not admit chirally cosmetic surgery.
\end{proposition}
\begin{proof}
This is a consequence of estimates of $a_2,v_3,a_4$ and $\det$ established so far.
\begin{align*}
&\frac{7a_2(K)^2-a_2(K)-10a_4(K)}{4v_3(K)}\\
&\quad \leq \frac{7a_2^2(K)}{4v_3(K)} & (\because \eqref{eqn:a2n-positive}, \eqref{eqn:v3-positive}) \\
&\quad \leq 7a_2(K) & (\because \mbox{Lemma } \ref{lemma:stoimenow-improvement}) \\
&\quad \leq  \frac{7}{8}c(D)^{2} & (\because \mbox{Proposition } \ref{proposition:estimate-conway})\\
&\quad\leq \frac{(1-d)\gamma^{(\btw(D)-7)/4}}{tw(D)}c(D)^2 & (\because \mbox{Assumption})\\
&\quad\leq \frac{1}{2}\det(K) & (\because \mbox{Proposition } \ref{prop:det-density})\\
&\quad \leq \frac{1}{2}\det(K) +3g(K)-\frac{5}{2}
\end{align*}
Thus by Theorem \ref{theorem:obstruction} (ii), $K$ admits no chirally cosmetic surgeries.
\end{proof}

For later use we restate Proposition \ref{prop:non-large-density} in terms of a condition on $c(R)$.

\begin{corollary}
\label{cor:non-large-density-crossing}
Let $D$ be a reduced, twist-reduced special alternating diagram and let $R$ be its maximum twist region. If $c(R) \leq (tw(D)-1)\left( 1-\frac{7tw(D)}{8}\gamma^{(7-tw(D))/4}\right)$ then $K$ admits no chirally cosmetic surgeries. 
\end{corollary}
\begin{proof}
Since $c(R)+(tw(D)-1) \leq c(D) = \frac{c(R)}{d(D)}$, we have $tw(D)-1 \leq c(R) \frac{1-d(D)}{d(D)}$. Thus by assumption
\begin{align*}
d(D) & < \frac{d(D)}{1-d(D)} \frac{c(R)}{tw(D)-1} \leq 1-\frac{7tw(D)}{8}\gamma^{(7-tw(D))/4}
\end{align*}
so $K$ admits no chirally cosmetic surgeries.
\end{proof}

To study the remaining case where the density $d(D)$ is close to $1$, we need somewhat finer estimate of $a_2(K)$ and $v_3(K)$. To this end, we distinguish two cases of twist regions. We say that a twist region $R$ with $c(R)>1$ is \emph{incoherent} if two strands of $R$ are oppositely oriented like $\raisebox{-3mm}{\begin{picture}(100,30)
\put(0,0){\line(1,1){24}}
\put(10,14){\vector(-1,1){10}}
\put(14,10){\line(1,-1){10}}
\put(24,24){\line(1,-1){10}}
\put(24,0){\line(1,1){24}}
\put(38,10){\line(1,-1){10}}
\put(52,10){$\cdots$} 
\put(68,0){\vector(1,1){24}}
\put(68,24){\line(1,-1){10}}
\put(82,10){\line(1,-1){10}}
\end{picture}}$ 
Otherwise, when two strands of $R$ are oriented in the same direction like 
$\raisebox{-3mm}{\begin{picture}(100,30)
\put(24,24){\vector(-1,-1){24}}
\put(10,14){\vector(-1,1){10}}
\put(14,10){\line(1,-1){10}}
\put(24,24){\line(1,-1){10}}
\put(24,0){\line(1,1){24}}
\put(38,10){\line(1,-1){10}}
\put(52,10){$\cdots$} 
\put(68,0){\line(1,1){24}}
\put(68,24){\line(1,-1){10}}
\put(82,10){\line(1,-1){10}}
\end{picture}}$ we say that the twist region $R$ is \emph{coherent}.

\begin{lemma}
\label{lemma:coherent}
Let $D$ be a positive alternating diagram and put $c=c(D)$ and $d=d(D)$. Assume that the maximum twist region $R$ is coherent.
If
\[  \frac{(c(R)-1)(c(R)^{2}-2c(R))}{24c(R)^{3}} \geq X \]
for some $X>0$
then
\[ \frac{7a_2(K)^2-a_2(K)-10a_4(K)}{4v_3(K)} \leq  \frac{7c(D)}{64Xd(D)^3} %\leq 21c(D)tw(D)^{3} 
\]
\end{lemma}
\begin{proof}
In the following, we denote by $c'_R$ the maximum odd integer such that $c'_R \leq c(R)$ (so $c(R)-1 \leq c'_R \leq c(R)$). 
 
We take a base point $\ast$ near the maximum twist region $R$.
When the maximum twist region $R$ is coherent, in the Gauss diagram $G_D$, the arrows from $R$ forms a sub Gauss-diagram of the form
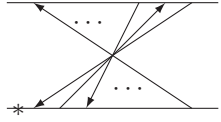
\[\begin{picture}(100,45)
\put(80,45){\line(-1,0){80}}
\put(2,2){$\ast$}
\put(0,5){\line(1,0){80}}
\put(70,45){\vector(-6,-4){60}}
\put(20,5){\vector(1,1){40}}
\put(50,45){\vector(-1,-2){20}}
\put(40,10){$\cdots$}
\put(25,35){$\cdots$}
\put(70,5){\vector(-6,4){60}}
\end{picture}
\]
In particular, $G_D$ contains a sub-Gauss diagram $G_{T_{2,c'_R}}$ where $T_{2,c'_R}$ is the standard $c'_R$-crossing diagram of the $(2,c'_R)$-torus knot. Thus by assumption
\begin{align*}
4v_3(K) & \geq v_3(T_{2,c'_R}) = \frac{c'_R(c'_R{}^{\!2}-1)}{24} \\
& \geq \frac{(c(R)-1)(c(R)^{2}-2c(R))}{24} & (\because c'_R\geq c(R)-1)\\
& \geq Xc(R)^{3} & (\because \mbox{Assumption})\\
\end{align*}
%Since {\color{dg} $c_R>10$}, $\frac{(c(R)-1)(c(R)^{2}-2c(R))}{24} \geq {\color{dg} X c(R)^{3}}$ so
%\begin{equation}
%\label{eqn:v3-estimate}
%4v_3(K) \geq {\color{dg} X c_R^{3}}
%\end{equation}
Therefore
\begin{align*}
\frac{7a_2(K)^2-a_2(K)-10a_4(K)}{4v_3(K)} 
&\leq \frac{7a_2(K)^2}{4v_3(K)} & (\because \eqref{eqn:a2n-positive}, \eqref{eqn:v3-positive}) \\
& \leq \frac{7c(D)^{4}}{64 X c(R)^{3}}& (\because \mbox{Proposition } \ref{proposition:estimate-conway}%\eqref{eqn:v3-estimate}
)\\
& = \frac{7c(D)}{64X d(D)^{3}} & (\because c(R)=c(D)d(D))\\
%& \leq 21c(D)tw(D)^{3} & (\because \mbox{Lemma }\ref{lemma:tw-density})
\end{align*}

\end{proof}

\begin{lemma}
\label{lemma:incoherent}
Let $D$ be a positive alternating diagram of a knot $K$ and put $c=c(D)$ and $d=d(D)$. Assume that the maximum twist region $R$ is incoherent and that
\[ \frac{16c(R)}{(c(R)-2)(c(R)-4)} \leq Y \]
for some $Y>0$.
Then
\[ \frac{7a_2(K)^2-a_2(K)-10a_4(K)}{4v_3(K)} \leq \frac{7+Y}{2}(1-d(D))^{2}c(D)^{2}\]
\end{lemma}
\begin{proof}

We take a base point $\ast$ near the maximum twist region $R$.
When the maximum twist region $R$ is coherent, in the Gauss diagram $G_D$, the arrows from $R$ forms parallel arcs as 
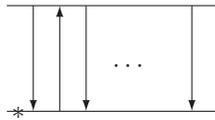
\[\begin{picture}(100,45)
\put(80,45){\line(-1,0){80}}
\put(2,2){$\ast$}
\put(0,5){\line(1,0){80}}
\put(10,45){\vector(0,-1){40}}
\put(20,5){\vector(0,1){40}}
\put(30,45){\vector(0,-1){40}}
\put(70,45){\vector(0,-1){40}}
\put(40,20){$\cdots$}
\end{picture}
\]
Thus the sub-Gauss diagram of $G_D$ from two or three arrows from $R$ does not contribute to $a_2(K)$ or $v_3(K)$.

For a positive alternating diagram $D$ that represents a knot, the \emph{base diagram} $D_{base}$ with respect to the twist region $R$ is the knot diagram having the following properties;
\begin{itemize}
\item[(i)] $D_{base}$ is the same as $D$ except at the twist region $R$.
\item[(ii)] $D_{base}$ is twist-reduced and represents a knot.
\item[(iii)] The number of crossings in the twist region $R$ is minimum among all diagrams satisfying (i) and (ii). 
\end{itemize}
Roughly speaking, $D_{base}$ is a knot diagram obtained from $D$ by removing crossings in the twist region $R$ as possible, preserving the twist-reducedness.
In particular, $D_{base}$ is also a positive and alternating diagram.
Let $K_{base}$ be the knot represented by $D_{base}$. 
%Then
%\begin{equation}
%a_2(K_{base})\geq g(K_{base}) = g(K) \geq \frac{c(K)}{4}
%\end{equation}

Let $\ell$ be the number of left arrows that intersects with arrows from $R$. Since $D$ is reduced and positive, $\ell> 0$. 

First we note that since all arrows from the twist regions are parallel, in the Gauss diagram formula \eqref{eqn:a2-formula} sub-Gauss diagram from two arrows in the twist region has no contribuition.

Since the number of the left arrows that forms a Gauss diagram $
\raisebox{-4mm}{
\begin{picture}(24,28)
\put(12,14){\circle{24}}
\put(9,0){$\ast$}
\put(21,22){\vector(-1,-1){17}}
\put(3,22){\vector(1,-1){17}}
%\put(10,26){\vector(1,-1){14}}
\end{picture} 
}
$
together with a right arrow coming from the crossings in the twist region $R$ is $\ell$, we get
\begin{equation}
\label{eqn:a2-base}
a_2(K) \leq a_2(K_{base}) + \ell \frac{c(R)}{2}
\end{equation}

Similarly, by counting the number of sub-Gauss diagram 
$
\raisebox{-4mm}{
\begin{picture}(24,28)
\put(12,14){\circle{24}}
\put(9,0){$\ast$}
\put(21,22){\vector(-1,-1){17}}
\put(3,22){\vector(1,-1){17}}
\put(10,26){\vector(1,-1){14}}
\end{picture} 
}
$
in the Gauss diagram formula of $4v_3$ so that two paralell right arrow $
\raisebox{-4mm}{
\begin{picture}(24,28)
\put(12,14){\circle{24}}
\put(9,0){$\ast$}
%\put(21,22){\vector(-1,-1){17}}
\put(3,22){\vector(1,-1){17}}
\put(10,26){\vector(1,-1){14}}
\end{picture} 
}
$ comes from the crossings in the twist region $R$ that does not belong to $D_{base}$ (there are at most one such crossings) we get

\begin{equation}
\label{eqn:v3-base}
4v_3(K) \geq v_3(K_{base}) + \ell \binom{c(R)/2-1}{2} 
\end{equation}

Thus
\begin{align*}
&\frac{7a_2(K)^{2}-a_2(K)-10a_4(K)}{4v_3(K)} \\
&\quad \leq \frac{7a_2(K)^{2}}{4v_3(K)} & (\because \eqref{eqn:a2n-positive}, \eqref{eqn:v3-positive}) \\
&\quad  \leq 7\frac{(a_2(K_{base}) + \ell c(R)/2)^2}{4v_3(K)} & (\because \eqref{eqn:a2-base})\\
&\quad =  7 \left( \frac{a_2(K_{base})^2}{4v_3(K)} + \frac{a_2(K)\ell c(R) + c(R)^2/4}{4v_3(K)}  \right) \\
&\quad  \leq 7 \left(  \frac{a_2(K_{base})^2}{4v_3(K_{base})} + \frac{a_2(K_{base})\ell c(R) + c(R)^2/4}{ \ell \binom{c(R)/2 -1}{2}} \right) & (\because \eqref{eqn:v3-base})\\
&\quad  \leq 7 \left(a_{2}(K_{base}) + a_2(K_{base})\frac{8c(R)}{(c(R)-2)(c(R)-4)} + \frac{c(R)}{\ell(c(R)-2)}\right)& (\because \mbox{Proposition } \ref{proposition:estimate-conway})\\
\end{align*}
Since 
\begin{align*}
&a_2(K_{base})\frac{8c(R)}{(c(R)-2)(c(R)-4)} + \frac{c(R)}{\ell(c(R)-2)} = a_2(K)\frac{c(R)}{(c(R)-2)}\left( \frac{8}{c(R)-4} + \frac{1}{\ell a_2(K)}\right)\\
&\quad \leq a_2(K_{base})\frac{c(R)}{(c(R)-2)}\left( \frac{8}{c(R)-4} + \frac{4}{c(R)}\right)\\
&\quad \leq a_2(K_{base})\frac{16c(R)}{(c(R)-2)(c(R)-4)}\\
&\quad \le Ya_2(K_{base})
\end{align*}
we get 
\begin{align*}
&\frac{7a_2(K)^{2}-a_2(K)-10a_4(K)}{4v_3(K)} \leq (7+Y)a_2(K_{base}) \\
&\quad  \leq \frac{7+Y}{8}c(D_{base})^2  &(\because \mbox{Proposition } \ref{proposition:estimate-conway})\\
&\quad  \leq \frac{7+Y}{8}\bigl( (1-d(D))c(D)+3\bigr) ^{2} &(\because c(D_{base}) \leq c(D)-c(R)+3 )\\
& \quad \leq   \frac{7+Y}{4}(1-d(D)^{2})c(D)^{2} % &(\because 3 \leq (1-d(D))c(D))
\end{align*}
Here the last inequality follows from 
\[ 3 \leq (\sqrt{2}-1)8 \leq (\sqrt{2}-1)(tw(D)-1) \leq (\sqrt{2}-1)(1-d(D))c(D) \]
since we assume $tw(D)\geq 9$.
\end{proof}

These estimate completes the proof of our main theorem.

\begin{proof}[Proof of Theorem \ref{theorem:main}]
Let $D$ be a reduced, twist-reduced positive alternating diagram of $K$ such that $tw(D)=\btw(D)=tw(K)$. 

Let $R$ be the maximum twist region. By Proposition \ref{prop:non-large-density} and Corollary \ref{cor:non-large-density-crossing} we may assume that
\begin{equation}
\label{eqn:1-d-bound}
1-d(D) < \frac{7tw(D)}{8}\gamma^{(3-tw(D))/4} 
\end{equation}
and that
\begin{equation}
\label{eqn:cR-bound}
c(R) > (tw(D)-1)\left( 1-\frac{7tw(D)}{8}\gamma^{(7-tw(D))/4}\right)
\end{equation}
Since we are assuming that $tw(D)>63$, by
\eqref{eqn:1-d-bound} and \eqref{eqn:cR-bound} $d(D)>\frac{9}{16}$, $c(R)>31$.
Consequently, $X>\frac{1}{28}$ and $Y<\frac{1}{4}$. Thus under the assumption that $tw(D)>63$ the following two inequalities hold.
\begin{equation}
\label{eqn:X}
\frac{7X}{64}\left(1 - \frac{7tw(D)}{8}\gamma^{(3-tw(D))/4}\right)^{-3} \leq \frac{(tw(D)-1)\gamma^{(\btw(D)-7)/4}}{tw(D)}.
\end{equation}
\begin{equation}
\label{eqn:Y}
\frac{7+Y}{4}\frac{7tw(D)}{8}\gamma^{(3-tw(D))/4} \leq \frac{\gamma^{(\btw(D)-7)/4}}{tw(D)}.
\end{equation}
If $R$ is coherent
\begin{align*}
& \frac{7a_2(K)^2-a_2(K)-10a_4(K)}{4v_3(K)} \leq \frac{7X c(D)}{64d(D)^{3}} & (\because \mbox{Lemma } \ref{lemma:coherent})\\
& \quad \leq \frac{7X}{64}\left(1 - \frac{7tw(D)}{8}\gamma^{(3-tw(D))/4}\right)^{-3}\!c(D)
& (\because \eqref{eqn:1-d-bound})\\
& \quad \leq \frac{(tw(D)-1)\gamma^{(\btw(D)-7)/4}}{tw(D)}c(D) & (\because \eqref{eqn:X})\\
& \quad \leq \frac{(1-d(D))\gamma^{(\btw(D)-7)/4}}{tw(D)}c(D)^2 & (\because \mbox{Lemma }\ref{lemma:tw-density})\\
& \quad < \frac{1}{2}\det(K)+3g(K)-\frac{5}{2} & (\because \mbox{Proposition }
\ref{prop:det-density})
\end{align*}

Similarly, when $R$ is incoherent then
\begin{align*}
&\frac{7a_2(K)^2-a_2(K)-10a_4(K)}{4v_3(K)} \\
& \quad \leq \frac{7+Y}{4}(1-d(D))^{2}c(D)^{2} & (\because \mbox{Lemma }\ref{lemma:incoherent})\\
& \quad \leq \frac{7+Y}{4}(1-d(D))\frac{7tw(D)}{8}\gamma^{(3-tw(D))/4}c(D)^{2} & (\because \eqref{eqn:1-d-bound})\\
&  \quad\leq \frac{(1-d(D))\gamma^{(\btw(D)-7)/4}}{tw(D)}c(D)^2 & (\because \eqref{eqn:Y})\\
& \quad < \frac{1}{2}\det(K)+3g(K)-\frac{5}{2}& (\because \mbox{Proposition } \ref{prop:det-density})
\end{align*}

Therefore $K$ does not admit chirally cosmetic surgery by Theorem \ref{theorem:obstruction} (ii).
\end{proof}

\section{Alternating knots}
\label{section:alternating}

\subsection{Knot invariant arguments}

At many points, our argument uses that the diagram $D$ is positive. Nevertheless,  we can use similar arguments to exclude chirally cosmetic surgeries for many alternating knots if add additional assumptions.

As a demonstration we observe that when the crossing number is not large compared with twist number (so the density is small) and $v_3(K)\neq 0$, then $K$ has no chirally cosmetic surgery.

\begin{proposition}
\label{proposition:alt}
Let $D$ be a reduced, twist reduced alternating diagram of a knot $K$ such that $4v_3(K)\neq 0$. If $c(D)^4 \leq 3\gamma^{tw(D)-1}$ then $K$ admits no chirally cosmetic surgery.
\end{proposition}
\begin{proof}
Since $4v_3(K)\neq 0$, $|4v_3(K)|\geq 1$ so 
\begin{align*}
& \left|\frac{7a_2(K)^2-a_2(K)-10a_4(K)}{4v_3(K)}\right|\\
& \qquad \leq |7a_2^{2}(K)| + |a_2(K)| + |10a_4(K)| \\
& \qquad \leq 8|a_2(K)|^2 + 10|a_4(K)| \\
& \qquad \leq \frac{1}{8}c(D)^{4} + \frac{10}{48}c(D)^{4} - 10 \binom{\frac{c(D)}{2}}{4} & (\because \mbox{ Proposition } \ref{proposition:estimate-conway})\\ 
& \qquad < \frac{1}{3}c(D)^{4} - \frac{c(D)}{2} \\
& \qquad \leq \gamma^{tw(D)-1} -\frac{c(D)}{2} & (\because \mbox{Assumption})\\
 \end{align*}
 On the other hand, since $|\sigma(K)| \leq 2g_4(K) \leq 2g_3(K) \leq c(D)-1$, by Theorem \ref{theorem:det-key}
\begin{align*}
\frac{1}{2}\left(\det(K)-|\sigma(K)|-1\right) > \gamma^{tw(D)-1} - \frac{c(D)}{2}.
 \end{align*}
Therefore $K$ does not admit chirally cosmetic surgery by Theorem \ref{theorem:obstruction} (i).
\end{proof}

%For example, when each twist region contains at most $M$ crossings for $M$ and $tw(K)$ is sufficiently large compared with $M$, then $K$ has no chirally cosmetic surgeires.
%Proposition \ref{proposition:alt} if the density is very small (close to $0$) and $tw(K)$ is sufficiently large then $K$ does not admit chirally cosmetic surgeries.

As is clear from the proof, for a general alternating case, $\left|\frac{7a_2(K)^2-a_2(K)-10a_4(K)}{4v_3(K)}\right|$ may have an order of $c(D)^{4}$. So our estimate of $\det(K)$ in Theorem \ref{theorem:det-key} is insufficient.

However, we point out if the density is large (close to $1$) and the maximum twist region $R$ is coherent then we have an estimate analogous to Lemma \ref{lemma:coherent}.

\begin{lemma}
\label{lemma:coherent-alternating}
Let $D$ be an alternating diagram. Assume that the maximum twist region $R$ is coherent, $c(R)>5$ and that $-2(1-d(D))^{3}+d(D)^{3}>0$. Then $v_3(K)\neq 0$ and
\[ \left| \frac{7a_2(K)^2-a_2(K)-10a_4(K)}{4v_3(K)} \right| < \frac{64}{-2(1-d(D))^{3}+d(D)^{3}}c(D) \]
\end{lemma}
\begin{proof}
If necessary, by taking the mirror image we may assume that the maximum twist region consists of positive crossings. 
Let $D_{base}$ be the base diagram as taken in the proof of Lemma \ref{lemma:incoherent}.
Then by a similar argument as Lemma \ref{lemma:coherent}, from by counting the sub-Gauss diagram that comes from arrows in $D_{base}$ and arrows in the coherent twist region $R$ we get
\begin{align*}
4v_3(K) & \geq 4v_3(K_{base}) +  v_3(T_{2,c'_R}) \\
& \geq 4v_3(K_{base}) +  \frac{(c(R)-1)(c(R)^{2}-2c(R))}{24} \\
& >   4v_3(K_{base}) + \frac{c(R)^3}{48} & (\because c(R)>5)\footnotemark \\
& \geq -\frac{c(D_{base})^{3}}{24} +   \frac{c(R)^{3}}{48} & (\because \mbox{ Proposition } \ref{prop:v3-bounds})\\
& \geq \frac{-2(1-d(D))^{3}+d(D)^{3}}{48}c(D)^{3} \  (>0)
\end{align*}

Therefore
\begin{align*}
 \left|\frac{7a_2(K)^2-a_2(K)-10a_4(K)}{4v_3(K)}\right| & \leq \frac{1}{4v_3(K)}\left(|7a_2^{2}(K)| + |a_2(K)| + |10a_4(K)|\right) \\
 & \leq \frac{1}{4v_3(K)} \left( 8|a_2(K)|^2 + 10|a_4(K)| \right)\\
 & < \frac{1}{4v_3(K)}\left(\frac{1}{3}c(D)^{4}\right)\\
 & \leq \frac{64}{-2(1-d(D))^{3}+d(D)^{3}}c(D)
 \end{align*}
\end{proof}
\footnotetext{
Here we use quite crude estimate. By a slightly more careful estimate as we did in Lemma \ref{lemma:coherent}, one can improve the estimate.}

Using this estimate instead of Lemma \ref{lemma:coherent}, a similar argument shows that when both $tw(D)$ and $d(D)$ are sufficiently large, then $K$ has no chirally cosmetic surgeries. Here we give one concrete sufficient condition. 

\begin{corollary}
\label{cor:alternating}
Let $D$ be a reduced, twist-reduced diagram of an alternating knot $K$.
If the maximum twist region $R$ of $D$ is coherent, $4v_3(K) \neq 0$, $tw(D)>19$, and $d(K)>\frac{2}{3}$, then $K$ does not admit chirally cosmetic surgeries. 
\end{corollary}
\begin{proof}
By Theorem \ref{theorem:det-key}, when $tw(K)>19$ then 
\[ \det(L) > 2\gamma^{-1}\left(\gamma^{\btw(D)} + (c(D)-\btw(D))\gamma^{(\btw(D)-1)/2} \right) > 2c(D)\gamma^{(\btw(D)-3)/2}. \]
Thus
\begin{align*}
\frac{1}{2}\left(\det(K)-|\sigma(K)|-1\right) & > \frac{1}{2}\det(K)-\frac{1}{2}c(D) \\
& > \left( \gamma^{(tw(D)-3)/2}-\frac{1}{2} \right)c(D) \\
& > \left(\gamma^{8}-\frac{1}{2}\right)c(D)  >  \frac{128}{9}c(D) & (\because tw(D)>19)\\
& >  \frac{64}{-2(1-d(D))^{3}+d(D)^{3}}c(D) & (\because d(D)>\frac{2}{3})\\
& >  \left|\frac{7a_2(K)^2-a_2(K)-10a_4(K)}{4v_3(K)}\right| & (\because \mbox{Lemma } \ref{lemma:coherent})
\end{align*}
Therefore $K$ does not admit chirally cosmetic surgery by Theorem \ref{theorem:obstruction} (i). 
\end{proof}
%For the case the maximum density is large and the maximum twist region is coherent we refer to \cite[Theorem 6.4]{iis1} which leads to similar but weaker non-existence result.

\subsection{Discussion and comparison with hyperbolic geometry argument}

We close the paper by a short discussion and comparison of an alternative approach based on hyperbolic geometry.

For a slope $s$ of hyperbolic knot $K$, let $L(s)$ be the \emph{normalized length} defined by  $L(s) = \frac{\ell(s)}{\sqrt{\mathrm{Area}(\partial C)}}$, where $\partial C$ is a cusp torus and $\ell(s)$ is the euclidian length of $s$. 
Let $\mathrm{Sys}(E(K))$ be the systole length of the knot complement $E(K)$, the length of the shortest closed geodesic. 

In \cite{bhw} it is pointed out that Thurston's hyperbolic Dehn surgery theorem and Mostov rigidity imply that $s$ and $s'$ surgeries on non-amphicheiral hyperbolic knot $K$ cannot be chirally cosmetic whenever the length of slopes $s$ and $s'$ are sufficiently large. In \cite{fps} they proved the following effective version of this criterion.

\begin{theorem}\cite[Theorem 7.30]{fps}
\label{theorem:hyperbolic-criterion}
Let $K$ be a hyperbolic knot and $s,s'$ be different slopes. If
%\begin{equation}
%\label{eqn:fps}
\[ L(s),L(s') \geq \max\left\{10.1, \sqrt{\frac{2\pi}{\mathrm{Sys}(E(K))} + 58}\right\}, \]
%\end{equation} 
then $s$ and $s'$ surgeries are not chirally cosmetic, unless $s=-s'$ and $K$ is amphichiral. 
\end{theorem}
Furthermore, they gave an explicit and computable finite set of candidates  of chirally cosmetic surgeries that allows us to test whether a given hyperbolic knot admit chirally cosmetic surgeries or not (see \cite[Theorem 1.13]{fps} for details).

On the other hand, for hyperbolic alternating knots, the cusp area is bounded below by the twist number.

\begin{theorem}\cite[Theorem 1.1]{lp}
\label{theorem:cusp-area}
For a prime alternating knot $K$ other than $(2,p)$ torus knots, the area of the maximal cusp $C$ of $E(K)$ satisfies 
%\begin{equation}
%\label{eqn:area} 
\[ \mathrm{Area}(\partial C) > A(tw(K)-2)\]
%\end{equation}
where $A$ is some constant, $A>2.278 \times 10^{-19}$.
\end{theorem} 
Since for every non-meridional slope $s$, $\mathrm{Area}(\partial C) \leq 3 \ell(s)$ (see \cite{lp}) the theorem leads to a lower bound of normalized length in terms of the twist numbers
%\begin{equation}
%\label{eqn:length} 
\[ L(s) =  \frac{\ell(s)}{\sqrt{\mathrm{Area}(\partial C)}} > \frac{\sqrt{\mathrm{Area}(\partial C)}}{3} > \frac{\sqrt{A(tw(K)-2)}}{3} \]
%\end{equation}
 
Thus, Theorem \ref{theorem:hyperbolic-criterion} and Theorem \ref{theorem:cusp-area} show that non-amphicheiral alternating knots have no chirally cosmetic surgeries, as long as $tw(K)$ is large and $\mathrm{Sys}(E(K))$ is not small. 

\begin{corollary}
\label{cor:hyperbolic}
Let $K$ be a prime alternating knot $K$. If $tw(K)>4 \times 10^{21}$ and $\mathrm{Sys}(E(K))>0.15$, then $K$ has no chirally cosmetic surgeries unless $K$ is amphichieral. 
\end{corollary}

Compared with Theorem \ref{theorem:main}, the required twist number is very large, due to the smallness of the constant $A$ in Theorem \ref{theorem:cusp-area}.
The constant $A$ can be improved when we add additional assumptions.
For example, when we assume that each twist region has at most $N$ crossings, then the constant $A$ can be taken as $\frac{1.844 \times 10^{-4}}{3N-1}$ \cite[Theorem 2.9]{lp}. 
Similarly, if we further assume that $K$ is a two-bridge knot, then the constant $A$ can be taken $\frac{8\sqrt{3}}{147} = 9.426 \times 10^{-2}$ \cite[Theorem 4.8]{fkp}. Thus for hyperbolic two-bridge knot case, Corollary \ref{cor:hyperbolic} holds under the weaker condition that $tw(K)>9800$ (and $\mathrm{Sys}(E(K))>0.15$).

We emphasize that the assumption on systole is crucial since $\mathrm{Sys}(E(K))$ can be arbitrary small even if we assume that $tw(K)$ is large -- the length of a crossing circle, a circle enclosing two strands of the twist region $R$ tends to zero as the crossing number $c(R)$ grows. 

This makes a sharp contrast with Corollary \ref{cor:alternating} which treats the case where the crossings are concentrated in a single crossing region (so the systole is arbitrary small), though we need several additional assumptions. Thus the hyperbolic geometry method and our knot invariant constraint method have quite different features and range of applicabilities.

Unfortunately, both methods require additional assumptions, so extending Theorem \ref{theorem:main} for general alternating knots requires more effective new constraints, or, substantial refinements of current arguments and estimates.

\end{document}